\documentclass[reqno,A4paper]{amsart}

\usepackage{amsmath}
\usepackage{amssymb}
\usepackage{amsthm}
\usepackage{enumerate}
\usepackage{mathrsfs} 
\usepackage{eqlist}
\usepackage{array}

\theoremstyle{plain}
\newtheorem{theorem}{Theorem}[section]
\newtheorem{prop}[theorem]{Proposition}

\newtheorem{corol}{Corollary}[theorem]

\theoremstyle{definition}
\newtheorem{definition}{Definition}
\newtheorem{remark}{\textup{Remark}} 
\newtheorem{example}{\textit{Example}} 

\usepackage{xcolor}
\usepackage{graphicx}
\usepackage{color}

\def\r{\mathbb R}
  
\def\C{\mathbb C}
\def\d{\mathbb D}

\numberwithin{equation}{section}

\begin{document}

\title[Differential geometry of curves in dual space]%
{Differential geometry of curves in dual space}
\author[Rafael L\'opez]%
{Rafael L\''opez*}

\newcommand{\acr}{\newline\indent}

\address{\llap{*\,}Departamento de Geometr\'{\i}a y Topolog\'{\i}a\acr
 Universidad de Granada 18071 Granada, SPAIN}
\email{rcamino@ugr.es} 



\subjclass[2010]{Primary 53A10; Secondary  53C21, 53C42} 
\keywords{dual space, Frenet frame, curvature, torsion}

\begin{abstract}We introduce the Frenet theory of curves in dual space $\d^3$. After defining the curvature and the torsion of a curve, we classify all curves in dual plane with constant curvature. We also establish the fundamental theorem of existence in the theory of dual curves, proving that  there is a dual curve with prescribed curvature and torsion. Finally we classify all dual curves with constant curvature and torsion.
\end{abstract}

\maketitle

\section{Introduction and the problem of reparametrization} 

The study of curves   in Riemannian spaces is one of the main   topics in classical differential geometry. The geometry of these curves is studied thanks to its Frenet frame and the notion of curvature. In $3$-dimensional Euclidean space $\r^3$, the curvature and the torsion of a curve determine, up to rigid motions of $\r^3$, the shape of the curve. In this paper we extend this theory to curves in dual space $\d^3$.

 Dual numbers were  introduced  by William Clifford in 1873 \cite{cl}. The set of dual numbers is the algebra $\d=\{x+\varepsilon y\colon x,y\in\r\}$ where $\varepsilon^2=0$.    Eduard Study used   dual numbers to represent lines in $\r^3$ as follows. Let $\d^3=\r^3\times\varepsilon\r^3$ be the space of dual vectors. Let $L=p+\langle v\rangle$ denote a line of $\r^3$, where $p\in\r^3$ is a point of $L$ and $v\in\r^3$ is the direction vector of $L$, which we assume unitary, $|v|=1$. We associate  to $L$ the dual vector $v+\varepsilon v^*\in \d^3$, where $v^*=p\times v$ is the moment of $L$. Notice that $v^*$ does not depend on the choice of the point $p$ of $L$. Moreover,   $v$ and $v^*$ are orthogonal vectors \cite{st}. Also Study represented the relative   displacement and orientation of two lines in space by dual numbers. So, if $\theta$ is the angle   between their direction vectors  and $d$ is the shortest distance between them, both numbers   are formally written as the dual number $\theta+\varepsilon d$.   Thanks to this type of representations of lines in $\r^3$,  some properties of the kinematics of curves can be studied in terms of dual numbers. This has attracted the interest in CADG, robotics and analysis of the mechanic of mechanisms: see for example, \cite{an, be,di,du,fi,ha}.
 
In this paper, we study   smooth curves in $\d^3$ from the viewpoint of differential geometry. For a smooth curve, we mean a differentiable map $\gamma\colon I\subset\r\to\d^3$. If we write $\gamma$ in its real and dual part, then $\gamma(t)=\alpha(t)+\varepsilon \beta(t)$, where $\alpha,\beta\colon I\to\r^3$ are smooth curves and the derivative of $\gamma$ is $\gamma'(t)=\alpha'(t)+\varepsilon\beta'(t)$.  The metric in $\d^3$ is defined by extending    the Euclidean metric $\langle,\rangle $ of $\r^3$ by bilinearity and using $\varepsilon ^2=0$, 
$$\langle a_1+\varepsilon b_1,a_2+\varepsilon b_2\rangle =\langle a_1,b_1\rangle+\varepsilon (\langle a_1,b_2\rangle+\langle b_1,a_2\rangle),$$
where $a_i,b_i\in\r^3$, $i=1,2$.
The purpose of this paper is to give a theory of Frenet frame which it allows a definition of curvature and torsion for a curve in $\d^3$. Recently, there is a certain literature in the study of curves in $\d^3$ from this viewpoint. However, and in opinion of the author, some of  these studies   suffer   some confusion about what variable is being considered in the curve and we now explain in detail. 

The problem finds if one   considers curves in  dual space $\d^3$ of real variable  or curves in $\d^3$ whose variable is a dual number, that is, of dual parameter.  Both objects are different. In the first case we are in the situation of an immersion of an interval $I$ of $\r$ in $\d^3$.  This corresponds with a pair of curves of real variable in $\r^3$.   In the second case, since a dual parameter corresponds to a pair of real numbers, we have an immersion of a domain $\Omega$ of $\r^2$  in $\d^3$. In this case, a curve of dual parameter represents a pair of surfaces in $\r^3$.

In consequence, the concept of differentiability  used in each case is different. In the first situation,  differentiability is the usual notion of a function of real variable which   extends in a natural way  to the real and dual parts of the curve.  However, if we consider curves with dual parameter, the differentiability is completely different.  
The notion of differentiability for functions defined in domains of $\d$ follows the same steps that holomorphic functions and similar  Cauchy-Riemann equations appear when this theory is developed \cite{ve}.   

We can compare the last situation of curves with dual parameter if we replace  the set of dual numbers by the set of complex numbers $\C$, $z=x+iy\in \C$, $i^2=-1$. Differentiability of functions  between a domain of $\C$ and $\C^3$ means  that the functions are holomorphic. The study of holomorphic curves in $\C^3$ is now attracting mathematicians since the  papers of Calabi and Lawson \cite{ca1,ca2,la}. We also point out  a remarkable paper of Bryant   where he studied  holomorphic curves in  the $6$-sphere ${\mathbb S}^6$ by means of a type of holomorphic Frenet frame defining a curvature and torsion of holomorphic curves \cite{br}. Notice that curves in $\d^3$ with dual parameter are not the object of this paper.

Coming back again with curves in dual space $\d^3$, surely the source of misunderstanding finds in \cite{ve}, when the author   reparametrizes a curve in $\d^3$ of real variable in order to have unitary speed vectors. More precisely, following   \cite[p. 147]{ve},   the problem arises when a curve in $\d^3$ of real variable is not parametrized by arc-length. Let $\gamma(s)=\alpha(s)+\varepsilon \beta(s)$, $s\in I\subset\r$ a curve in $\d^3$ where $|\alpha'(s)|=1$ for all $s\in I$. However $|\gamma'(s)|=1+ \varepsilon\langle\alpha'(s),\beta'(s)\rangle$. By repeating here the   computations done in \cite{ve}, the author defines the so-called  `dual arc-length parameter' $\tilde{s}$ by 
\begin{equation}\label{pr}
\tilde{s}=\int_0^s|\gamma'(\sigma)|\, d\sigma=\int_0^s(1+\varepsilon \langle\alpha'(\sigma),\beta'(\sigma)\rangle\, d\sigma=s+\varepsilon \int_0^s\langle\alpha'(\sigma),\beta'(\sigma)\rangle\, d\sigma.
\end{equation}
But this identity does not define a dual variable because $\tilde{s}$ is a function of $s$, $\tilde{s}=\tilde{s}(s)$. In the simple case that $\alpha'$ and $\beta'$ are orthogonal, we  have $\tilde{s}=s+\varepsilon$ and $\tilde{s}(s)$ only describes  the line $y=0$ of $\d$. Other example shows the same problem. If, for instance,  $\gamma(s)=(\cos s,\sin s,0)+\varepsilon (s,0,s)$, then $\tilde{s}=s+\varepsilon \sin s$, which it does not cover $\d$ and thus  it is not a  dual variable.   This confusion has been followed by  many authors   attempting a Frenet  theory for functions that  initially   were of a real variable and next, after \eqref{pr}, these functions  were   transformed of dual variable: see for example, \cite{cg,lcj,lp,yac1,yac2}.

In order to establish a theory of Frenet frame for curves in dual space $\d^3$   the problem finds in those dual curves that are   not parametrized by arc-length.   The issue is that not  every dual curve  $\gamma(t)=\alpha(t)+\varepsilon\beta(t)$ can be reparametrized by arc-length and extra   conditions to $\alpha$ and $\beta$ must be assumed. This is explained at the beginning of Sect. \ref{sec-2}. The conditions of a dual curve to have reparametrizations by arc-length are the following:
 \begin{enumerate}
\item The real part $\alpha$ of the curve must be a regular curve.
\item The velocities vectors $\alpha'(t)$ and $\beta'(t)$ must be orthogonal for all $t\in I$.
\end{enumerate}

Although the second condition may be restrictive, we have shown that orthogonality is natural in the context of kinematics for the description of lines because the director vector $v$ of a line and its moment $v^*$ are orthogonal vectors.  Therefore, the theory of Frenet curves in $\d^3$ will be developed under both conditions. 

The organization in this paper is the following. In Sect. \ref{sec-2}, we show the problem of reparametrizations by arc-length and we give how to solve up. Once the objects of study have been well established, we define the Frenet frame of a curve in dual space $\d^3$. This allows to give the definition of curvature and torsion. With both concepts, it is natural to propose similar questions motivated by the classical theory of curves in Euclidean space. In Sect. \ref{sec-3}, we classify the curves with constant curvature in the dual plane $\d^2$, proving that the real part is a circle and the dual part is a curve of spiral type (Thm. \ref{t2}).   In Sect. \ref{sec-4}, we establish a theorem of existence of dual curves when the curvature and the torsion are prescribed.  More precisely, we prove that given two differentiable functions $k_0$ and $\tau_0$ defined in some interval of $\r$, there exists a dual curve of curvature and torsion $k_0$ and $\tau_0$, respectively (Thm. \ref{t-41}). We show some explicit examples in Sect. \ref{sec-5} as for examples, curves with zero torsion. Uniqueness cannot be obtained as in Euclidean space, that is, non-congruent curves in $\d^3$ in the Euclidean sense, may have the same curvature and torsion. Finally in Sect. \ref{sec-6}, we obtain explicitly the parametrizations of  all dual curves with constant curvature and constant torsion (Thm. \ref{t51}).

A last remark.  Along this paper we will call {\it dual curves} to refer curves of real variable in the dual space $\d^3$. A more appropriate terminology may be a {\it curve in dual space $\d^3$}   such as it appears in the title of this paper. A dual curve would be a term employed for curves with dual parameter. However, it is common  in the literature  the use of  terms as spherical curves, hyperbolic curves or  Lorentzian curves to indicate curves of real variable whose image is include in spheres, hyperbolic spaces or Lorentzian spaces, respectively.  For this reason we also employ the terminology of dual curve as curve in dual space.
 
\section{The Frenet frame of dual curves}\label{sec-2}

In this section, we assign a Frenet frame for a curve in dual space and, consequently, we will define the curvature and the torsion of a curve. For basic properties of dual numbers, we refer to the reader to \cite{ve}. This first part of this section holds for curves of real parameter on $n$-dimensional space $\d^n$. 

We  show the problem of reparametrization by arc-length of a dual curve. Let $\gamma\colon I\subset \r\to\d^n$  be a differentiable map,   $\gamma=\gamma(t)$.  Let us write $\gamma=\alpha+\varepsilon\beta$, where $\alpha,\beta\colon I\to\r^3$ are the real and dual parts of $\gamma$, respectively. Then  
$\gamma'(t)=\alpha'(t)+\varepsilon\beta'(t)$. For convenience, we drop the variable of $\gamma$ if it is understood. We will assume that $\gamma$ is a regular curve, that is, $\gamma'(t)\not=0$ for all $t\in I$.   In order to define the Frenet frame, we need that the curve is parametrized by arc-length. We say that $\gamma$ is parametrized by arc-length if $\langle\gamma'(t),\gamma'(t)\rangle =1$ for all $t\in I$. Since 
$$
\langle \gamma'(t),\gamma'(t)\rangle=\langle\alpha'(t),\alpha'(t)\rangle+2\varepsilon\langle\alpha'(t),\beta'(t)\rangle,
$$
this implies
\begin{equation}\label{eq1}
\langle \alpha'(t),\alpha'(t)\rangle=1,\quad  \langle\alpha'(t),\beta'(t)\rangle=0,\quad t\in I.
\end{equation}
Therefore     necessary conditions  for a curve $\gamma$ to be parametrized by arc-length is that $\alpha$ is a regular curve and that the velocities of the real and dual parts must be orthogonal.  If the curve is not parametrized by arc-length, then   we reparametrize $\gamma$ in order to get it.  As usually,   a reparametrization of $\gamma$ is a change of the parameter $t$ by $s$ by means of a  diffeomorphism  $\phi\colon J\subset \r \to I$  between intervals of $\r$, $t=\phi(s)$. A reparametrization of $\gamma$ is  the curve   $\tilde{\gamma}(s)=\gamma(\phi(s))$, $s\in J$. If we impose the condition that   $|\tilde{\gamma}'(s)|=1$ for all $s\in J$,  the chain rule gives 
\begin{equation}\label{eq2}
\tilde{\gamma}'(s)=\phi'(s)\gamma'(\phi(s))=\phi'(s)\left(\alpha'(\phi(s))+\varepsilon\beta'(\phi(s))\right).
\end{equation}
Using that $\phi'\not=0$, the condition $\langle \tilde{\gamma}'(s),\tilde{\gamma}'(s)\rangle=1$ is equivalent to 
$$   \phi'(s)^2\langle\alpha'(\phi(s)),\alpha'(\phi(s))\rangle=1,\quad \langle \alpha'(\phi(s)),\beta'(\phi(s))\rangle=0.$$
As a consequence, we see that a reparametrization of $\gamma$ allows to get unitary tangent vector for the real part of $\gamma$ if $\alpha'\not=0$, that is, if $\alpha$ is regular. However, the orthogonality condition   \eqref{eq1} remains invariant and it must be assumed from the beginning.  The study of dual curves from a differential geometry viewpoint makes necessary both conditions. We summarize the above calculations in the following result.
 
 \begin{prop}\label{pr1}
 Let $\gamma\colon I\to\d^n$ be a  curve, $\gamma(t)=\alpha(t)+\varepsilon\beta(t)$. If $\alpha$ is a regular curve and $\langle\alpha'(t),\beta'(t)\rangle=0$ for all $t\in I$, then there is a reparametrization of $\gamma$ by arc-length. 
\end{prop}
\begin{example} The following     curves cannot be parametrized by arc-length.
\begin{enumerate}
\item $\gamma(t)=(0,0,1)+\varepsilon (t,0,0)$. The real part is not a regular curve even though the dual part is a regular curve.
\item $\gamma(t)=(t,0,0)+\varepsilon (t,t,0)$. The vectors $\alpha'(t)=(1,0,0)$ and $\beta'(t)=(1,1,0)$ are not orthogonal.
\end{enumerate}
 
\end{example}

\begin{example} Let $\gamma(t)=(\cos t^2,\sin t^2,0)+\varepsilon (0,0,\cos t)$, $t\in I=(0,\infty)$. We have 
$\gamma'(t)=2t(-\sin t^2,\cos t^2,0)-\varepsilon (0,0,\sin t)$. The real part $\alpha$ is regular because $|\alpha'(t)|=2t$. Moreover,   $\langle\alpha',\beta'\rangle=0$. Thus $\gamma$ satisfies the condition of Prop. \ref{pr1}.  We reparametrize $\alpha$ by arc-length obtaining $\phi(s)=\sqrt{s}$.  Then the desired reparametrization of $\gamma$ is 
$$\tilde{\gamma}(s)=\gamma(\sqrt{s})=(\cos s,\sin s,0)+\varepsilon (0,0,\cos (\sqrt{s})).$$
\end{example}

Proposition \ref{pr1} says that if a curve $\gamma$   is not parametrized by arc-length, then reparametrizing only its real part, we have    a reparametrization by arc-length of the dual curve $\gamma$. This reparametrization of $\gamma$ also reparametrizes the dual part $\beta$.  However it is not possible, in general, that  the curve $\beta$ is parametrized by arc-length, that is, there is not  a common parameter $s$ which it is the arc-length parameter for   both curves $\alpha$ and $\beta$. Only  in this case, we have the following definition. 

\begin{definition}
 A     curve $\gamma$ in the dual space $\d^n$ parametrized by arc-length is said to be normalized if  $\beta$ is parametrized by arc-length.
 \end{definition}

We begin with the Frenet theory for dual curves. Let $\gamma$ be a curve in $\d^n$ parametrized by arc-length.  Define the {\it tangent vector} of $\gamma$ as 
$$T(s)=\gamma'(s)=\alpha'(s)+\varepsilon\beta'(s).$$
As usually, the unit normal vector of $\gamma$ is defined by differentiating    $\langle T(s),T(s)\rangle=1$, obtaining   $\langle \gamma''(s),T(s)\rangle=0$. As in the Euclidean case, the normal vector is defined once we  discard the case  $\gamma''(s)=0$ for all $s\in I$.

\begin{prop} Let $\gamma$ be a curve in $\d^n$ parametrized by arc length. Then  $\gamma''(s)=0$ for all $s\in I$ if and only if $\gamma$ is a straight line of $\d^n$. Moreover, the straight lines $\alpha$ and $\beta$ are orthogonal.
\end{prop}

\begin{proof}
Since $\gamma''(s)=\alpha''(s)+\varepsilon\beta''(s)$,  the condition $\gamma''(s)=0$ is equivalent to $\alpha''(s)=0$ and $\beta''(s)=0$ everywhere, that is,   $\alpha$ and $\beta$ are straight lines in $\r^3$, hence the result. The last statement is immediate.
\end{proof}

From now, we will consider curves in $3$-dimensional dual space $\d^3$. Let $\gamma$ be a dual curve in $\d^3$ such that  $|\gamma''(s)|\not=0$ everywhere.   Define the {\it normal vector} of $\gamma$ by 
$$N(s)=\frac{1}{|\gamma''(s)|}\gamma''(s).$$

\begin{definition} If $\gamma$ is a dual curve in $\d^3$ parametrized by arc-length, the {\it   curvature} of $\gamma$ is defined by 
$$\kappa(s)=|\gamma''(s)|.$$
\end{definition}
  Finally we define the binormal vector and the torsion.

\begin{definition}  If $\gamma$ is a dual curve  in $\d^3$ parametrized by arc-length, the {\it   binormal vector} $B$ of $\gamma$ is defined by $B(s)=T(s)\times N(s)$ and the torsion $\tau$ by 
$$\tau=\langle N',B\rangle.$$
\end{definition}
Here $\times$ is the cross product in $\d^3$ which is defined by bilinearity from the cross product of $\r^3$ \cite[p. 142]{ve}. Notice that $|B|=1$ and $B$ is orthogonal to $T$ and $N$. The basis $\{T,N,B\}$ is   orthonormal and it is called the {\it Frenet frame} of $\gamma$. It is also immediate the Frenet equations
\begin{equation}\label{ef}
\begin{split}
T'(s)=&  \kappa(s) N(s),\\
N'(s)=& -\kappa(s)T(s)+\tau(s)B(s),\\
B'(s)=& \qquad -\tau(s)N(s).
\end{split}
\end{equation}

We can express the normal and binormal vectors of $\gamma$ as well as its curvature and torsion in terms of that of $\alpha$ because $\alpha$ is also parametrized by arc-length. We   denote the Frenet frame of $\alpha$ by  $\{T_\alpha,N_\alpha,B_\alpha\}$ and $\kappa_\alpha$ and $\tau_\alpha$ stand for its curvature and torsion respectively. The following result is a consequence of   straightforward computations.

\begin{prop} Let $\gamma=\alpha+\varepsilon\beta$ be a curve in $\d^3$ parametrized by arc-length. Then  
\begin{equation}\label{k}
\begin{split}
T&=T_\alpha+\varepsilon \beta',\\
N&=N_\alpha+\frac{\varepsilon}{\kappa_\alpha}\left(\beta''-\langle \beta'',N_\alpha \rangle N_\alpha\right),\\
B&= B_\alpha+\frac{\varepsilon}{\kappa_\alpha}\left(T_\alpha\times\beta''+\kappa_\alpha \beta'\times N_\alpha-\langle \beta'',N_\alpha\rangle B_\alpha\right),\\
\kappa&=\kappa_\alpha+\varepsilon\langle \beta'',N_\alpha\rangle,\\
\tau&=\tau_\alpha+ \varepsilon\frac{1}{\kappa_\alpha}\left(\langle\beta'''-\frac{\kappa_\alpha'}{\kappa_\alpha}\beta''+\kappa_\alpha^2\beta',B_\alpha\rangle-\tau_\alpha\langle\beta'',N_\alpha\rangle\right).
\end{split}
\end{equation}
\end{prop}

 In case that $\gamma$ is a normalized curve, the vectors $\beta'$, $\beta''$ and $\beta'''$ can be expressed in terms of the Frenet frame   $\{T_\beta,N_\beta,B_\beta\}$ of $\beta$ and of its curvature $\kappa_\beta$ and torsion $\tau_\beta$.

\begin{corol} \label{c-n}
Let $\gamma=\alpha+\varepsilon\beta$ be a normalized curve in $\d^3$. Then
\begin{equation*}
\begin{split}
T&=T_\alpha+\varepsilon T_\beta,\\
 N&=N_\alpha+\varepsilon \frac{\kappa_\beta}{\kappa_\alpha}\left( N_\beta-\langle N_\alpha,N_\beta\rangle N_\alpha\right),\\
 B&=B_\alpha+\frac{\varepsilon}{\kappa_\alpha}\left(\kappa_\beta T_\alpha\times N_\beta+\kappa_\alpha T_\beta\times N_\alpha-\kappa_\beta \langle N_\alpha,N_\beta\rangle B_\alpha\right),\\
 \kappa&=\kappa_\alpha+\varepsilon\kappa_\beta\langle N_\alpha,N_\beta\rangle,\\
 \tau&=\tau_\alpha+\varepsilon\frac{1}{\kappa_\alpha} \langle(\kappa_\alpha^2-\kappa_\beta^2)T_\beta+\kappa_\alpha(\frac{\kappa_\beta}{\kappa_\alpha})'N_\beta+\kappa_\beta\tau_\beta B_\beta,B_\alpha\rangle\\
 &-\varepsilon\frac{\tau_\alpha}{\kappa_\alpha}\kappa_\beta\langle N_\beta,N_\alpha\rangle
 \end{split}
 \end{equation*}
 \end{corol}

 \begin{example}  Let $\gamma(s)=(\cos s,\sin s,0)+\varepsilon (0,0,s)$. Then $\gamma$ is normalized and we have 
\begin{equation*}
\begin{split}
T(s)&=(-\sin s,\cos s,0)+\varepsilon (0,0,1),\\
 N(s)&=(-\cos s,-\sin s,0),\\
 B(s)&=(0,0,1)+\varepsilon (\sin s,-\cos s,0).
 \end{split}
 \end{equation*}
Then  $\kappa=1 $. Since $N'(s)=(\sin s,-\cos s,0)$,  we have  $\tau=\langle N',B\rangle=\varepsilon$.  This curve has constant  curvature and  torsion. Moreover, $\kappa$ is a real number and $\tau$ is a pure dual number.  
\end{example}

\begin{example} Let $\alpha$ be a curve in $\r^3$ parametrized by arc length and consider the dual curve $\gamma(s)=\alpha(s)+\varepsilon \alpha'(s)$. Since $\gamma'=\alpha'+\varepsilon \alpha''$, then   $\gamma$ is parametrized by arc-length because $\langle\alpha',\alpha''\rangle=0$. A computation of the Frenet frame gives 
\begin{equation*}
\begin{split}
T&=T_\alpha+ \varepsilon T_\alpha',\\
N&=N_\alpha+ \varepsilon N_\alpha',\\
B&=B_\alpha+ \varepsilon B_\alpha'.
\end{split}
\end{equation*}
It is then immediate 
\begin{equation*}
\begin{split}
\kappa&=\kappa_\alpha+ \varepsilon \kappa_\alpha',\\
\tau&=\tau_\alpha+ \varepsilon \tau_\alpha'.
\end{split}
\end{equation*}
\end{example}

\begin{example} Let $\gamma(s)=\alpha(s)+\varepsilon\alpha'(s)$, where $\alpha$ is the circular helix $\alpha(s)=(r\cos(s/m),r\sin (s/m), hs/m)$, where $r,h\in\r$, $r>0$ and $m=\sqrt{r^2+h^2}$. We know  $\kappa_\alpha=r/m^2$ and $\tau=h/m^2$. For the dual curve $\gamma(s)=\alpha(s)+\varepsilon \alpha'(s)$ we have $\kappa=\kappa_\alpha$ and $\tau=\tau_\alpha$. 
 \end{example}

\begin{remark} Formulas \eqref{k}  show that the dual curvature and dual torsion have   a   behaviour with respect to dilations similar as in Euclidean space. To be precise, let $\lambda>0$. Given a dual curve $\gamma$, consider
$$\tilde{\gamma}(s)=\lambda \gamma(\frac{s}{\lambda})= \tilde{\alpha}(s)+\varepsilon \tilde{\beta}(s) =\lambda \alpha(\frac{s}{\lambda})+\varepsilon\lambda\beta(\frac{s}{\lambda}).$$
Then $\tilde{\gamma}$ is also parametrized by arc-length. 
Since $\tilde{\beta}^{(n)}(s)=\frac{1}{\lambda^{n-1}}\beta^{(n)}(s/\lambda)$, then it is immediate 
$$ \kappa_{\tilde{\gamma}} =\frac{1}{\lambda}\kappa_\gamma ,\quad\tau_{\tilde{\gamma}} =\frac{1}{\lambda}\tau_\gamma. $$
\end{remark}

\section{Planar curves with constant curvature}\label{sec-3}

Once we have established a theory of Frenet frame for curves in dual space, it is natural to give results of classification according the curvature and the torsion. For example, we can cask for those curves with constant curvature or with constant torsion. In Euclidean space, there are not explicit parametrizations of curves with constant curvature or curves with constant torsion. See for example \cite{ba,mo,we}. In the case of dual curves, the problem is more difficult. For instance, if the dual curvature $\kappa$ is constant, then the real part of the curve has constant curvature, which it is already a problem in Euclidean space. Moreover, we need to find the dual part of the curve.

However, and as in the Euclidean space, we can restrict the problem to curves in the dual plane $\d^2$. If in the Euclidean plane the classification is well known (circles), in this section we solve the analogous problem in $\d^2$ determining explicitly the dual part, which it will be a type of spiral curve.

The theory developed until here for curves in dual space $\d^3$ also holds for curves in the dual plane $\d^2$ after the convenient changes. In such a case, there is no a notion of binormal vector and  torsion.    We find the classification of the planar curves with constant curvature in the following result.

\begin{theorem}\label{t2}
 Let $\gamma\colon I\to\d^2$ be a   curve parametrized by arc-length. If its curvature $\kappa$ is constant, then up to a rigid motion of $\r^2$,   $\alpha$ is a circle of radius $r>0$,  
\begin{equation}\label{p1}
\alpha(s)= r(\cos\frac{s}{r},\sin\frac{s}{r}).
\end{equation}
 Moreover, the dual part $\beta$ is parametrized by 
\begin{equation}\label{p2}
\begin{split}
\beta(s)&= -ar^2(\cos\frac{s}{r},\sin\frac{s}{r})+r(as+b)(-\sin\frac{s}{r},\cos\frac{s}{r})\\
&=r\left(-c_1\alpha(s)+(c_1 s+c_2)\alpha'(s)\right),
\end{split}
\end{equation}
where $a, b\in\r$ are constants.
\end{theorem}

\begin{proof} Since $\beta'$ is orthogonal to $\alpha'$, there is a function $\lambda=\lambda(s)$ such that $\beta'(s)=\lambda(s)N_\alpha(s)$. In particular, $\beta''=-\lambda\kappa_\alpha T_\alpha+\lambda' N_\alpha$. Thus the computation of the curvature yields 
$$\kappa=\kappa_\alpha+\lambda'\varepsilon.$$
We deduce that $\kappa$ is constant if and only if $\kappa_\alpha$ and $\lambda'$ are constant functions. Therefore $\kappa_\alpha=1/r$ for some $r>0$ and for $\lambda$, we have   $\lambda(s)=as+b$, $a,b\in\r$.   In particular, $\alpha$ is a circle of radius $r$. After a rigid motion of $\r^2$, we can reparametrize $\alpha$ by \eqref{p1}. Since $N_\alpha(s)= -(\cos\frac{s}{r},\sin\frac{s}{r})$, we have 
$$\beta'(s)=\lambda(s) N_\alpha(s)=-(as+b)(\cos\frac{s}{r},\sin\frac{s}{r}).$$
 An integration of this equation gives \eqref{p2}. 
 
\end{proof}
 The curve $\beta$ falls in  a known family of planar curves of spiral type. Notice that 
  $\beta$ in \eqref{p2} is not parametrized by arc-length because $|\beta'(s)|=|as+b|$. A reparametrization of $\beta$ by arc-length is 
$\beta(-\frac{b+\sqrt{b^2-2as}}{a})$, which we denote by $\beta(s)$ again.  Then   the curvature of $\beta$ is
\begin{equation}\label{b2}
\kappa_\beta(s)=|\beta''(s)|= \frac{1}{r\sqrt{b^2-2as}}.
\end{equation}
 In general, planar curves   in Euclidean plane $\r^2$  are called log-aesthetic curves of slope $\mu$ if the curvature $\kappa(s)$ is given by 
$$\kappa(s)=(as+b)^\mu,\quad \mu\not=0$$
for some constants $a$ and $b$ \cite{ha,ys}. These curves include clothoids ($\mu=1$) and log-spirals ($\mu=-1$). In the case of the curve $\beta$ of Thm. \ref{t2}, the value of $\mu$ is $\mu=-1/2$. See Fig. \ref{fig0}

 \begin{figure}[hbtp]
\centering
\includegraphics[width=.4\textwidth]{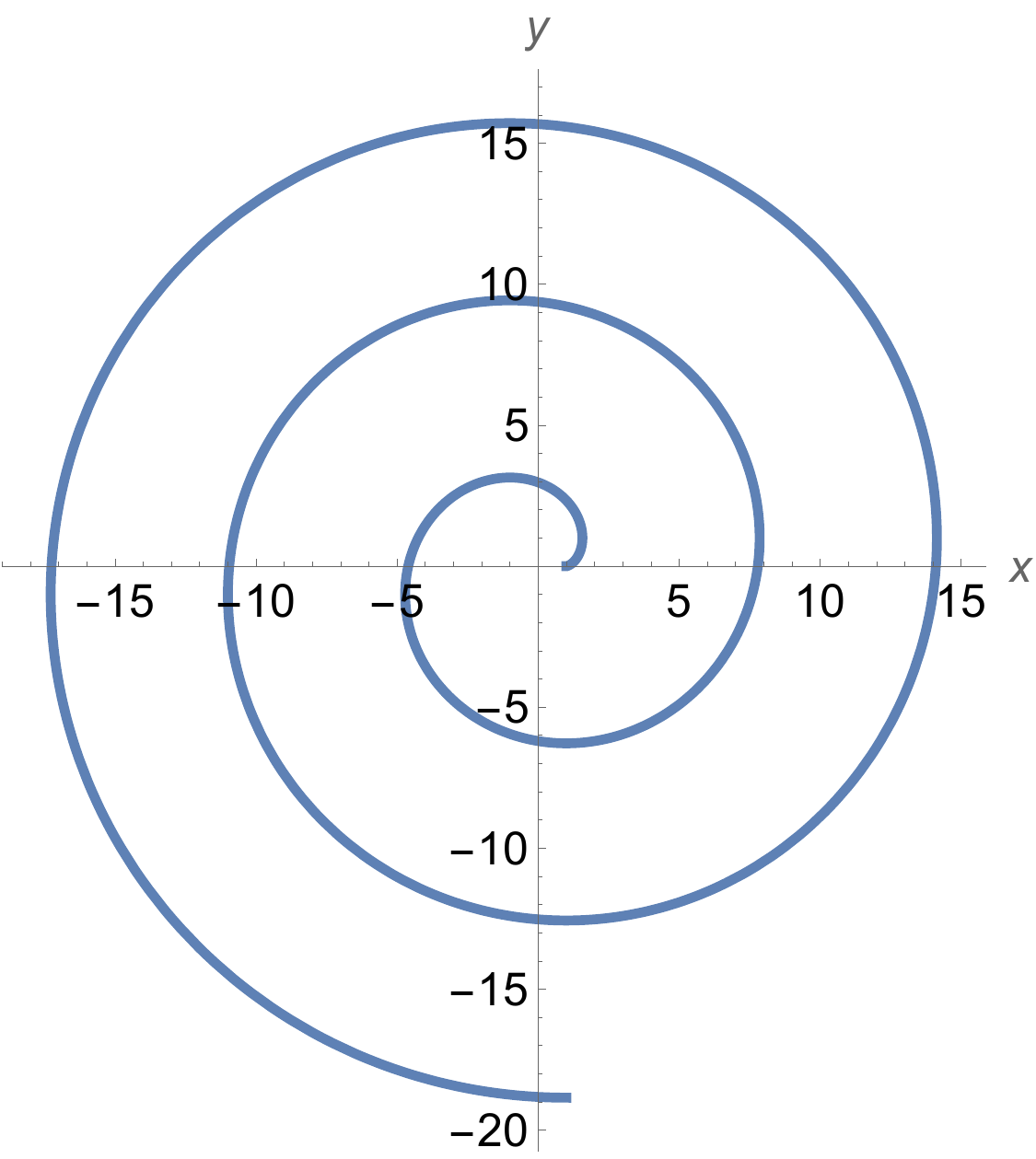}
\caption{The curve $\beta$ of Thm. \ref{t1}. Here $r=-a=1$, $b=0$}\label{fig0}
\end{figure}

\begin{remark} \label{rem-uniqueness}
Theorem \ref{t2} shows that  the   curve $\alpha$ is unique after a rigid motion of $\r^2$ because this is a consequence of the theory of Euclidean plane curves. However, the curve $\beta$ is not unique up to Euclidean congruences. This is due to the existence of constants $a$ and $b$ in \eqref{p2}.  
 \end{remark}

\section{Theorem of existence of dual curves}\label{sec-4}

In the theory of curves in Euclidean space it is known that given two smooth functions $k_0,\tau_0\colon I\to\r$, with $k_0>0$, there is a curve in $\r^3$ whose curvature  is $k_0$ and torsion  is $\tau_0$. Moreover, two curves with the same curvature and torsion coincide up to a rigid motion of $\r^3$. It is natural to ask if the same situation occurs for dual curves. We have pointed out in Rem. \ref{rem-uniqueness}  that uniqueness in the Euclidean sense is not true, that is, if two curves $\gamma_i(s)=\alpha_i(s)+\varepsilon\beta_i(s)$, $i=1,2$, in $\d^2$ have the same curvature, they are not necessarily congruent in the sense that there is a rigid motion of $\r^3$ carrying $\alpha_1$ in $\alpha_2$ and $\beta_1$ in $\beta_2$.   The same observation holds   for curves in dual space $\d^3$.

The   question  of existence is the following: given  two smooth functions   $k_0,\tau_0\colon I\to\d$, with $\textrm{Re}(k_0)>0$, is there a curve in $\d^3$ whose curvature and torsion are $k_0$ and   $\tau_0$, respectively?

 Since the real parts of  $\kappa$ and $\tau$ in \eqref{k} coincide with the curvature and the torsion of the real part $\alpha$ of     $\gamma$, then   $\alpha$ can be found by the classical theory of Euclidean curves. However, the problem acquires an extra  difficulty in finding the dual part because the curve $\beta$ has the  property to be orthogonal to $\alpha$. We look the problem in the dual plane $\d^2$. 

Let $k_0=k_1+\varepsilon k_2$ be a smooth function in $\d$ such that $k_1>0$. We ask for the existence of a dual curve in $\d^2$ with curvature $k_0$. According to \eqref{k}, we must have $\kappa_\alpha=k_1$ and $k_2=\langle \beta'',N_\alpha \rangle$. Thus, $\alpha$ is a curve in $\r^2$ with curvature $k_1$. The existence of $\alpha$ is assured by standard theory. Once we have the curve $\alpha$, we know its normal vector  $N_\alpha$. It remains to solve the equation
\begin{equation}\label{kb}
\langle \beta'',N_\alpha \rangle=k_2.
\end{equation}
Notice that \eqref{kb} is an ODE whose solution is assured. However, this is not enough because it may occur that $\beta$ is not orthogonal to $\alpha$, as it is required. To be precise, since   \eqref{kb} is of second order, the solution has integration constants which must be chosen in order to ensure orthogonality of $\alpha'$ and $\beta'$.  We see an example.

\begin{example} Let $k_0=1+\varepsilon$. For the curve $\gamma=\alpha+\varepsilon\beta$, we know   $\alpha(s)=(\cos s,\sin s)$. Since $\beta'$ is orthogonal to $T_\alpha$, then $\beta'=\lambda N_\alpha$ for some function $\lambda$. Then $\beta''=\lambda'N_\alpha-\lambda T_\alpha$. Thus 
$\langle \beta'',N_\alpha \rangle=1$ writes as $\lambda '=1$ and this gives $\lambda(s)=s+c$, $c\in\r$. This implies 
$$\beta'(s)=-(s+c_2)(\cos s,\sin s),\quad c_2\in\r.$$ 
An integration of this equation gives the parametrization   \eqref{p2}.

\end{example}

This example shows the strategy to follows. In the case that $\gamma$ is a planar dual curve, $\gamma\colon I\to\d^2$, it is possible to give an explicit  constructive method to find the curve $\gamma$, answering to the initial question of this section. 

\begin{theorem}  \label{t-41} If $k_0\colon I\to\r$ is a smooth function  with $\textrm{Re}(k_0)>0$, then there exist a curve $\gamma\colon I\to\d^2$ whose curvature   is   $k_0(s)$.  The real part of $\gamma$ is unique up to a rigid motion of $\r^2$.
\end{theorem}

\begin{proof} 
If $k_0(s)=k_1(s)+\varepsilon k_2(s)$, using \eqref{k}, we require to solve the system
\begin{equation}\label{ex}
\begin{split}
 \kappa_\alpha&=k_1,\\
 \langle \beta'',N_\alpha\rangle&=k_2.
\end{split}
\end{equation}
From   the first equation, we consider a curve $\alpha\colon I\to\r^2$ whose curvature $\kappa_\alpha$ is $k_1$. This curve is determined up to a rigid motion of $\r^2$. The solution of the first equation is given   after initial conditions of $\alpha$ and $\alpha'$ at a given point  $s=s_0\in I$.  Furthermore, from the local theory of Euclidean planar curves, it is known that if we define
$$\theta(s)=\int_{s_0}^sk_1(u)\, du,$$
  the curve $\alpha$ defined by 
\begin{equation}\label{a2}
\alpha(s)=\left(\int_{s_0}^s\cos\theta(u)\, du, \int_{s_0}^s\sin\theta(u)\, du\right)
\end{equation}
has curvature $\kappa_\alpha=\theta'$. Once we know the curve $\alpha$ we are going  to find $\beta$. Because $\beta'$ is orthogonal to $\alpha'$, we write 
$$\beta'(s)=\lambda(s)N_\alpha(s),$$
 where $\lambda$ is determined by the second equation of \eqref{ex}. Indeed, $\beta''=\lambda' N_\alpha-\lambda k_1 \alpha'$. Thus the condition $\langle \beta'',N_\alpha\rangle=k_2$  is 
$$ \lambda'(s)=k_2(s).$$
This gives $\lambda(s)=\int^s_{s_0} k_2(u)\ du$. Finally the curve $\beta$ is obtained by solving the ODE
\begin{equation}\label{b5}
\beta'(s)=\left(\int^s_{s_0} k_2(u)\ du\right)N_\alpha(s).
\end{equation}
\end{proof}

Theorem \ref{t-41} provides an explicit method to find the curve $\gamma=\alpha+\varepsilon \beta$. The curve $\alpha$ is given in \eqref{a2} and $\beta$ from \eqref{b5}, namely,   
$$\beta(s)=-\int^s_{s_0}  \left(\int^u_{s_0}  k_2(w)\ dw\right)\left( -\sin\theta(u),\cos\theta(u) \right)\, du.$$
We show two examples. 

\begin{example} Let $k_0(s)=1/r+  2\varepsilon s$, with $r>0$. Then  $k_1(s)=1/r$ and $k_2(s)=2s$. By taking $s_0=0$, then $\theta(s)=s/r$. Now  $\alpha$ is a circle of radius $r$, namely, 
$\alpha(s)=r(\sin(s/r),-\cos(s/r))$. On the other hand, 
$$\int^u_0 k_2(w)\, dw= u^2+a,\quad a\in\r.$$ 
The computation of $\beta$ yields
\begin{equation*}
\begin{split}
\beta(s)&=r(s^2-2 r^2)( \cos  (\frac{s}{r}),\sin(\frac{s}{r}))+2r^2(  (r-s)  \sin \left(\frac{s}{r}\right),s \cos \left(\frac{s}{r}\right) )+ar(\cos(\frac{s}{r}),\sin(\frac{s}{r}))\\
&=r(s^2+2r^2+a)\alpha'(s)-2rs\alpha(s)+2r^3(\sin (\frac{s}{r}),0) .
\end{split}
\end{equation*}
After a reparametrization by arc-length of $\beta$, the curvature of $\beta$ is
$$\kappa_\beta(s)=\frac{1}{\sqrt[3]{9}r\, s^\frac23  }.$$
\end{example}

\begin{example} Let $k_0(s)=\frac{1}{s}+\frac{1}{s^2}\varepsilon$, in particular, $k_1(s)=\frac{1}{s}$, $k_2(s)=\frac{1}{s^2}$ and $\theta(s)=\log s$. A computation of $\alpha$ yields
$$\alpha(s)=\frac{s}{2}  \sin (\log (s))(1,1)+\frac{s}{2}\cos(\log (s))(1,-1).$$
Now $\int^u k_2(w)\, dw= -\frac{1}{u}+a$, $a\in\r$. The computation of $\beta$ gives
\begin{equation*}
\begin{split}
\beta(s)&=-\left(\cos (\log (s)),\sin (\log (s))\right)\\
&-\left(\frac{1}{2} a s (\cos (\log (s))+\sin (\log (s))),\frac{1}{2} a s (\sin (\log (s))-\cos (\log (s)))\right).
\end{split}
\end{equation*}
The curvature of $\beta$ is $\kappa_\beta(s)=\frac{1}{1+as}$. This proves that $\beta$ is a log-spiral. Notice that if $a=0$, then $\beta$ is a circle of radius $1$.  
\end{example}

We now consider the existence problem for dual curves in $\d^3$. The answer to the initial question is affirmative. 

\begin{theorem} \label{t1}
Given differentiable functions $k_0=k_0(s)$, with $\textrm{Re}(k_0)>0$ and $\tau_0=\tau_0(s)$, $s\in I\subset\r$, there exists a   curve $\gamma\colon I\to\d^3$, $\gamma=\gamma(s)$, such that $s$ is the arc length and $k_0$ and $\tau_0$ are the curvature and torsion of $\gamma$, respectively. Moreover, the real part of $\gamma$  is unique up to a rigid motion of $\r^3$.  
\end{theorem}

\begin{proof}
Let 
\begin{equation}\label{kt}
k_0(s)=k_1(s)+\varepsilon k_2(s) , k_1(s)>0,\qquad \tau_0(s)=\tau_1(s)+\varepsilon \tau_2(s).
\end{equation}
 If $\gamma=\alpha+\varepsilon \beta$ is a curve in $\d^3$ with curvature $k_0$ and torsion $\tau_0$,  by the formulas \eqref{k}, we deduce that $\alpha$ is a curve in $\r^3$ whose curvature and torsion are $\kappa_\alpha=k_1$ and     $\tau_\alpha=\tau_1$, respectively. This determines the real part of $\gamma$. Moreover, $\alpha$ is unique up to a rigid motion of $\r^3$, proving the last statement of theorem.

Following \eqref{k}, the curve $\beta$ that we are looking for satisfies
\begin{equation}\label{k2}
\left\{\begin{split}
k_2&=\langle \beta'',N_\alpha\rangle,\\
\tau_2&=\frac{1}{\kappa_\alpha}\left(\langle\beta'''-\frac{\kappa_\alpha'}{\kappa_\alpha}\beta''+\kappa_\alpha^2\beta',B_\alpha\rangle-\tau_\alpha\langle\beta'',N_\alpha\rangle\right).
\end{split}\right.
\end{equation}
Since $\beta'(s)$ is orthogonal to $T_\alpha(s)$ for all $s\in I$, then there are two functions $x=x(s)$ and $y=y(s)$ to determine such that
\begin{equation}\label{b3}
\beta'(s)=x(s) N_\alpha(s)+y(s) B_\alpha(s)).
\end{equation}
 A computation of   $\beta''$ and $\beta'''$ gives:
 \begin{equation*}
\begin{split}
\beta''&=-\kappa_\alpha x\, T_\alpha+(x'-\tau_\alpha y)N_\alpha+(x\tau_\alpha +y')B_\alpha,\\
\beta'''&=(-2 \kappa_\alpha  x'-x \kappa_\alpha '+\kappa_\alpha  \tau_\alpha  y)T_\alpha+(x''-x \left(\kappa_\alpha ^2+\tau_\alpha ^2\right)-2 \tau_\alpha  y'-y \tau_\alpha ')N_\alpha\\
&+(2 \tau_\alpha  x'+x \tau_\alpha '+y''-y\tau_\alpha ^2)B_\alpha.
\end{split}
\end{equation*}
 Computing \eqref{k2}, we have
  \begin{equation}\label{k3}
  \left\{
\begin{split}
k_2&=x'-y\tau_\alpha,\\
\kappa_\alpha^2\tau_2&= \kappa_\alpha y'' +\kappa_\alpha  \left(\tau_\alpha  x'+x \tau_\alpha '\right)-\kappa_\alpha ' \left(\tau_\alpha  x+y'\right)+\kappa_\alpha ^3 y.
\end{split}
\right.
\end{equation}
Notice that the system \eqref{k3} is linear. Since $\kappa_\alpha=k_1>0$, standard theory of ODE assures the existence of a unique solution  of \eqref{k3} after initial conditions are given. 
  If $x=x(s)$ and $y=y(s)$ are solutions of \eqref{k3}, the curve $\beta$ is obtained by solving \eqref{b3}. This concludes the proof.  
 \end{proof}
 
 \section{Examples}\label{sec-5}
The proof of Thm. \ref{t1}  gives a practical method to find the dual curve $\gamma$. Suppose that the functions $k_0$ and $\tau_0$ are given by \eqref{kt}.
\begin{enumerate}
\item Find the curve $\alpha$ such that $\kappa_\alpha=k_1$ and $\tau_\alpha=\tau_1$.
\item Compute the Frenet frame of $\alpha$.
\item Using the values of $k_2$ and $\tau_2$, find the solutions $x(s)$ and $y(s)$ of \eqref{k3}.
\item Using the above two steps, solve \eqref{b3}.
\end{enumerate}

In the following example, we show a curve where the torsion is a pure dual function.

\begin{example} Let $k_0(s)=1$ and $\tau(s)=\varepsilon s$. From   Thm. \ref{t1}, the curve $\alpha$ has curvature $1$ and torsion $0$. Thus $\alpha$ is a circle of radius $1$, say  $\alpha(s)=(\cos s,\sin s,0)$.  
Now $k_2=\tau_1=0$ and $\tau_2=s$. The system \eqref{k3} is
 \begin{equation*} 
   \left\{
\begin{split}
0&=x',\\
s&=   y'' +y +  y.
\end{split}
\right.
\end{equation*}
Then $N(s)=-\alpha(s)$ and $B(s)=(0,0,1)$. The solutions are  
 \begin{equation*}
 \left\{ 
\begin{split}
x(s)&=c_1 ,\\
 y(s)&=s+c_2 \cos(s)+c_3\sin(s),
 \end{split}\right.
\end{equation*}
where $c_i\in\r$ are constants.  Then the curve $\beta$ is 
$$\beta(s)=\left( -c_1 \sin (s),c_1 \cos (s),c_2 \sin (s)-c_3 \cos (s)+\frac{s^2}{2}\right).$$
\end{example}

We study curves with  with zero torsion.  In Euclidean space $\r^3$, these curves are characterized by the fact to be contained in planes of $\r^3$. For dual curves, the dual part cannot be found explicitly.

\begin{theorem} \label{t41}
If a dual curve $\gamma$ has zero torsion $\tau(s)=0$ for all $s\in I$, then   $\alpha$ is a planar curve and $\beta$ is determined by \eqref{b3}, where $x$ and $y$ satisfy the system 
  \begin{equation}\label{tt} 
  \left\{
\begin{split}
k_2&=x',\\
0&= \kappa_\alpha y''  -\kappa_\alpha '  y' +\kappa_\alpha ^3 y.
\end{split}
\right.
\end{equation}
\end{theorem}

\begin{proof}
We use the proof of Thm. \ref{t1}. Since $\tau=0$, then $\tau_\alpha=0$ and  the real part is a curve with zero torsion, hence, $\alpha$ is a planar curve. On the other hand, we know $\tau_2=0$ in the system \eqref{k3}, obtaining \eqref{tt}.
\end{proof}

\begin{example}\label{ex1}
We find a curve with torsion $\tau=0$ and curvature $\kappa=1-\varepsilon\sin s$.  We know     $\alpha(s)=(\cos s,\sin s,0)$. Then  $N_\alpha(s)=-(\cos s,-\sin s,0)$ and $B_\alpha(s)=(0,0,1)$. Since 
 $k_2(s)=-\sin s$,   the first equation of \eqref{tt} gives $x(s)=\cos s+c_1$, $c_1\in\r$. The second equation of \eqref{tt} is $y''+y=0$ whose solution is $y(s)=c_2\cos s+c_3\sin s$, $c_2,c_3\in\r$. Using the functions $x(s)$ and $y(s)$, the solution of \eqref{b3} is 
\begin{equation}\label{b6}
\beta(s)=\left(\begin{array}{c}
 -\frac{1}{2} (s+2 c_1\sin (s)+\sin (s) \cos (s))\\
 \frac{1}{2} \cos (s) (2 c_1+\cos (s))\\
 c_2 \sin (s)-c_3 \cos (s)\end{array}\right).
 \end{equation}
 In Figure \ref{fig3}, we show some pictures of this curve for different choices of the integration constants $c_i$.  If $c_2=c_3=0$, then $\beta$ is a planar curve contained in the $xy$-coordinate plane.

  \begin{figure}[hbtp]
\centering
\includegraphics[width=.4\textwidth]{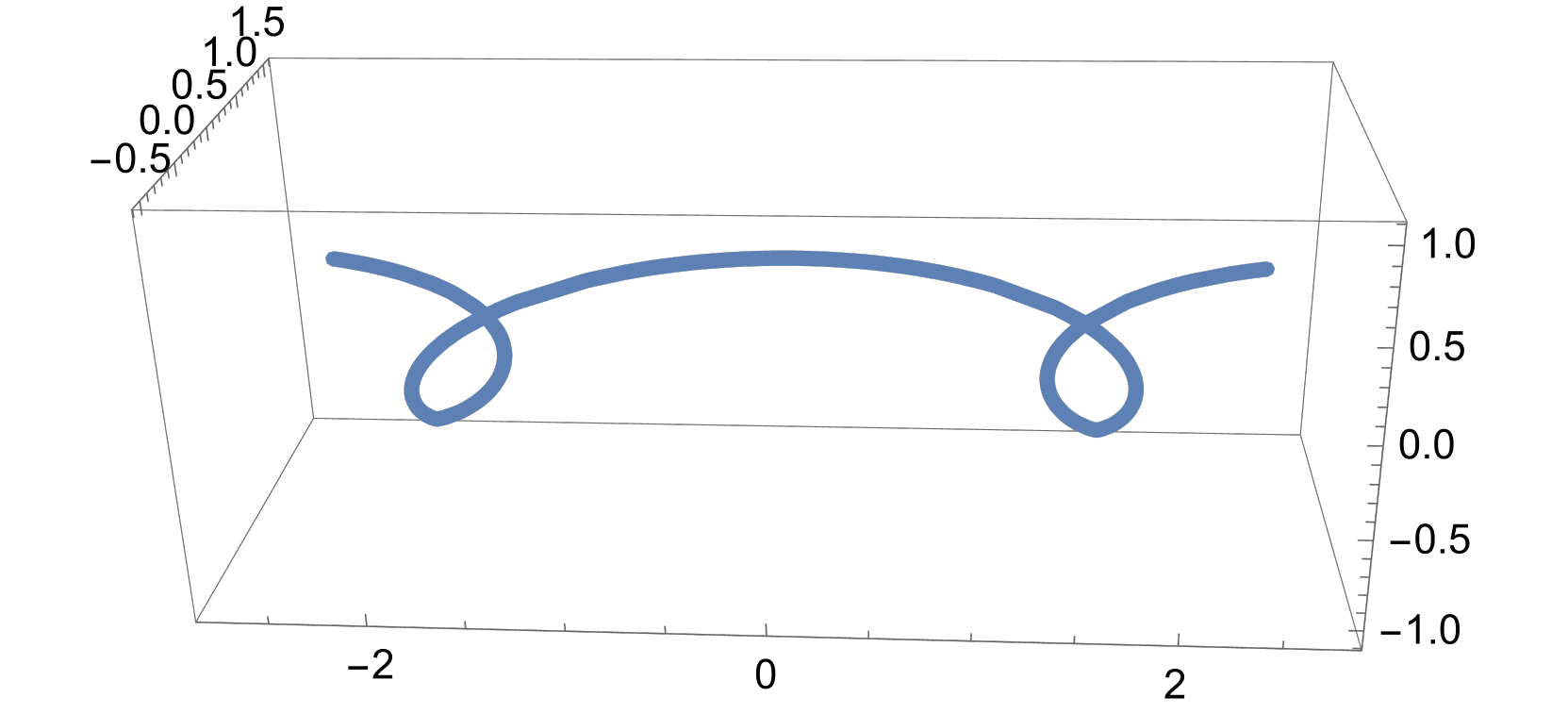}\  \includegraphics[width=.23\textwidth]{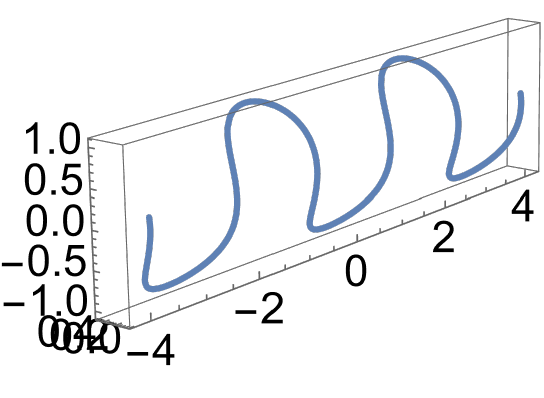}\  \includegraphics[width=.35\textwidth]{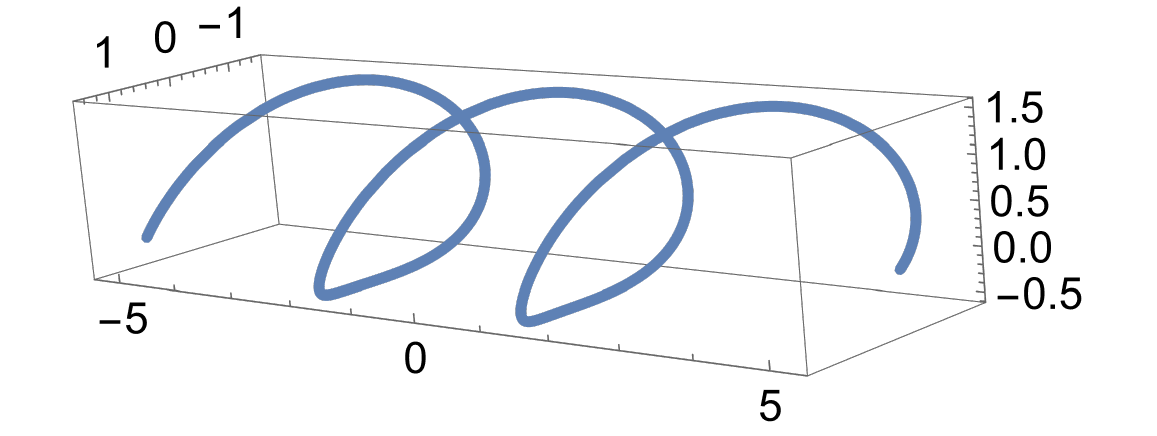}
\caption{Dual part $\beta$ of a dual curve with $\kappa=1-\varepsilon \sin s$ and $\tau=0$: see \eqref{b6} in Ex. \ref{ex1}.  Left: $(c_1,c_2,c_3)=(1,0,0)$, planar curve; middle: $(c_1,c_2,c_3)=(0,0,1)$; right $(c_1,c_2,c_3)=(1,1,1)$. }\label{fig3}
\end{figure}
 \end{example}

\section{Curves with constant curvature and torsion }\label{sec-6}

  We classify all dual curves with constant curvature and constant torsion. The key fact is that in such a case, the real part is a circular helix and the coefficients of the linear system \eqref{k3} are constants.

\begin{theorem} \label{t51}
Let $k_0=k_1+\varepsilon k_2$, $k_1>0$, and $\tau_0=\tau_1+\varepsilon\tau_2$ be constants. The dual curves $\gamma(s)=\alpha(s)+\varepsilon \beta(s)$ with curvature and torsion $k_0$ and $\tau_0$ respectively, are parametrized as follows. Let $r,h\in\r$, $r>0$, such that   
\begin{equation}\label{rh}
k_1=\frac{r}{m^2},\quad \tau_1=\frac{h}{m^2},\quad m=\sqrt{r^2+h^2}.
\end{equation}
Then the curve $\alpha$ is the helix of curvature $k_1$ and torsion $\tau_1$, which after a rigid motion is parametrized by 
$$\alpha(s)=(r\cos\left(\frac{s}{m}\right),r\sin \left(\frac{s}{m}\right),\frac{h}{m}s).$$
The curve $\beta$ is given by 
\begin{equation*}
\begin{split}
\beta(s)&=\cos \left(\frac{s}{m}\right)\left(-A,\frac{B}{m},-c_3 r m \right)+\sin \left(\frac{s}{m}\right)\left(-\frac{B}{m},-A,c_2 r\right) \\
&+ s\left(c_3 h,-\frac{c_2 h}{m},\frac{r}{m}(r\tau_2-hk_2)\right).
\end{split}
\end{equation*}
 where
\begin{equation*}
\begin{split}
A&=(r^2-h^2) k_2+2 r h \tau_2 ,\\
B&=rs(r k_2 + h  \tau_2)+m^2 (  c_3h+c_1).
\end{split}
\end{equation*}
\end{theorem}  
\begin{proof}
The curve $\alpha$ is obtained from \eqref{k} knowing that the curvature and the torsion of $\alpha$ are $k_1$ and $\tau_1$, respectively. 
With this data, we solve \eqref{k3}, obtaining explicit solutions, namely, 
\begin{equation*}
\begin{split}
x(s)&=\frac{r^2 k_2 s}{m^2}+h \left(\frac{r s \tau_2}{m^2}-c_3 \cos \left(\frac{s}{m}\right)+\frac{c_2}{m} \sin \left(\frac{s}{m}\right) + c_3\right)+c_1,\\
y(s)&=r \tau_2-h k_2+c_2 \cos \left(\frac{s}{m}\right)+m c_3 \sin \left(\frac{s}{m}\right),
\end{split}
\end{equation*}
where $c_i\in\r$ are constants of integration. Finally, the curve $\beta$ is obtained by integrating \eqref{b3}.\end{proof}

We particularize this theorem in two interesting cases:
\begin{enumerate}
\item $k_0\in\r$ and $\tau_0\in\varepsilon \r$   (Fig. \ref{fig2}, left).
\item $k_0\in\r$ and $\tau_0\in  \r$   (Fig. \ref{fig2}, right).
\end{enumerate}
In the next two results, $c_i$ stand for integration constants. 
\begin{corol} \label{cor52}
The only dual curves with constant curvature $\kappa=1/r>0$ and constant torsion $\tau=\tau_2\varepsilon$ are
\begin{equation*}
\begin{split}
\alpha(s)&=(r\cos\left(\frac{s}{r}\right),r\sin \left(\frac{s}{r}\right),0)\\
\beta(s)&=\cos \left(\frac{s}{r}\right)\left(0,c_1 r,-c_3 r^2 \right)+r\sin \left(\frac{s}{r}\right)(-c_1 ,0,c_2 ) + \left(0,0, r\tau_2s)\right).
\end{split}
\end{equation*}
\end{corol}

\begin{corol}\label{cor53}
 The only dual curves with constant curvature $\kappa=k_1>0$ and constant torsion $\tau=\tau_1\in\r$ are
\begin{equation*}
\begin{split}
\alpha(s)&=(r\cos\left(\frac{s}{m}\right),r\sin \left(\frac{s}{m}\right),\frac{h}{m}s),\\
\beta(s)&=\cos \left(\frac{s}{m}\right)\left(0,m(c_1+c_3h),-c_3 rm \right)+\sin \left(\frac{s}{m}\right)\left(-m(c_1+c_3h),0,c_2 r\right) \\
&+ hs(c_3 ,-\frac{c_2 }{m},0),
\end{split}
\end{equation*}
where $r$ and $h$ are given by \eqref{rh}. The curve $\beta$ is planar.
\end{corol}
\begin{proof}
A computation gives $\mbox{det}(\beta',\beta'',\beta''')=0$, which proves that $\beta$ is planar.
\end{proof}

 \begin{figure}[hbtp]
\centering
\includegraphics[width=.2\textwidth]{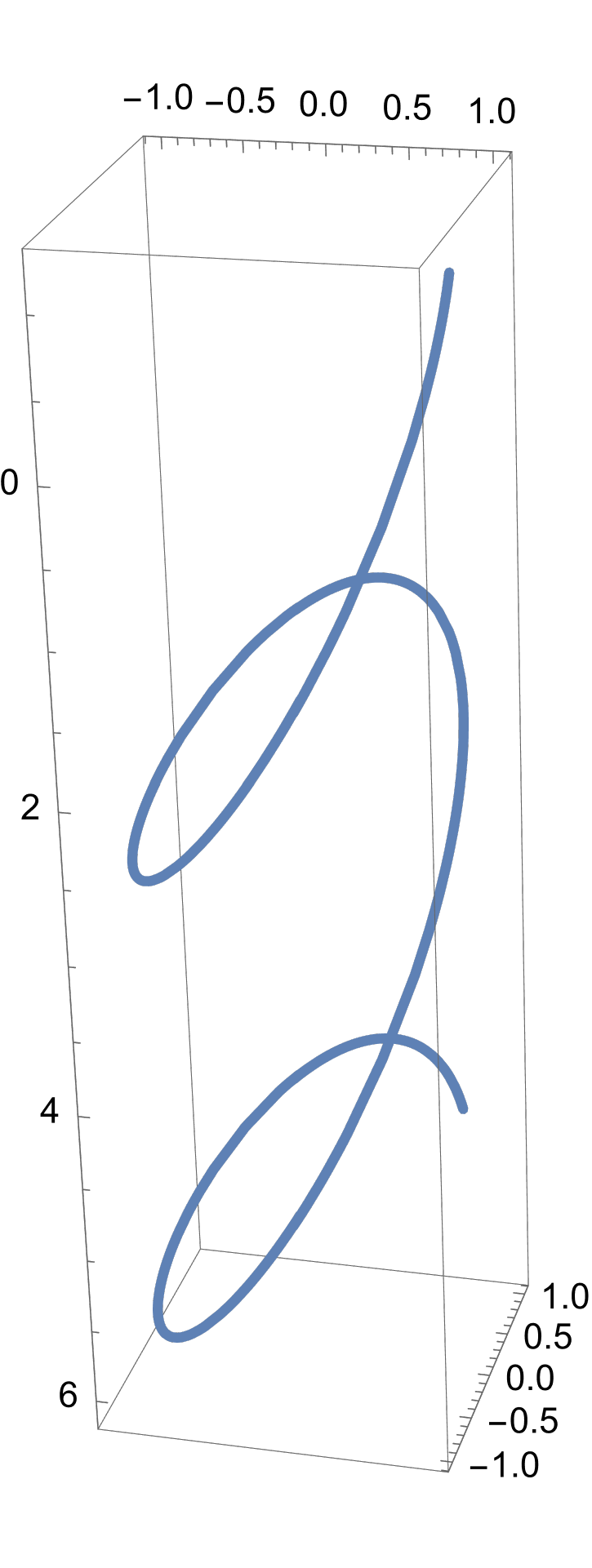}\hspace*{1cm}\includegraphics[width=.5\textwidth]{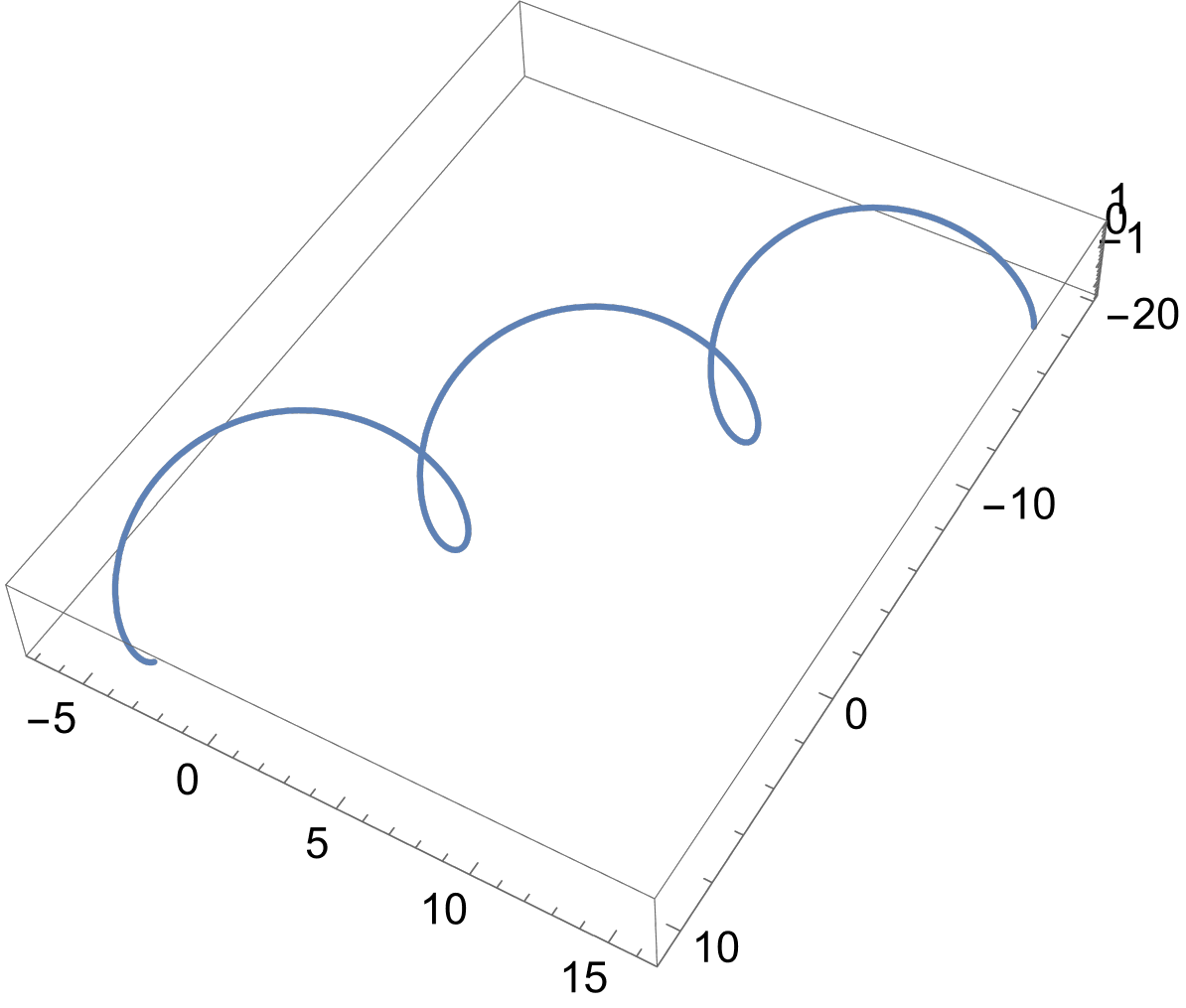}
\caption{Dual parts of curves of Cor. \ref{cor52} (left) and Cor. \ref{cor53} (right).  }\label{fig2}
\end{figure}

\section{Conclusions and outlook } 

This work presents the Frenet theory of curves of real parameter in dual space $\d^3$.   In the Introduction, we have shown     some misleading in the literature concerning to the reparametrization by arc-length in the sense that it is not possible to reparametrize a curve in $\d^3$ of real parameter and to convert it in a curve with dual   parameter. Specially, we have emphasized that the study of curves in $\d^3$ with dual parameter is completely different and, possibly, more  difficult according with its analogy of holomorphic curves.

For dual curves with real parameter,   not every dual curve can be repara\-metrized by arc-length and an orthogonality condition between the velocities of the real and dual parts has to be imposed.  The Frenet theory of dual curves allows to obtain   some extension of  some results  of Euclidean curves. In general, this type of results  will be more difficult than their analogs in Euclidean space because of the existence of the dual part as well as the orthogonality condition with the real part. This is clearly shown in Thm. \ref{t-41} when we assuming that the torsion of the curve is zero everywhere where it was not possible to determine explicitly the dual part of the curve. In contrast, this shows that the theory of dual curves is richer than its analogue of Euclidean spaces. 

The present work also confirms that  a theory of differential geometry of dual curves can be developed in future works. For example, it would be important to have a result of uniqueness     of Thm. \ref{t1} in terms of dual geometry.



 \section*{Acknowledgements} Rafael L\'opez is a member of the IMAG and of the Research Group ``Problemas variacionales en geometr\'{\i}a'',  Junta de Andaluc\'{\i}a (FQM 325). This research has been partially supported by MINECO/MICINN/FEDER grant no. PID2023-150727NB-I00,  and by the ``Mar\'{\i}a de Maeztu'' Excellence Unit IMAG, reference CEX2020-001105-M, funded by MCINN/AEI/10.13039/ 501100011033/ CEX2020-001105-M.

\end{document}